\documentclass[ECP]{ejpecp}

\SHORTTITLE{A note on tamed Euler approximations}

\TITLE{A note on tamed Euler approximations}

\AUTHORS{%
  Sotirios~Sabanis\footnote{University of Edinburgh, Scotland, UK.
    \EMAIL{s.sabanis@ed.ac.uk}}}


\AMSSUBJ{60H35}

\SUBMITTED{May 21, 2013}

\ACCEPTED{May 31, 2013}

\VOLUME{18}

\YEAR{2013}
\PAPERNUM{47}
\DOI{v18-2824}

\ABSTRACT{Strong convergence results on tamed Euler schemes, which
approximate stochastic differential equations with superlinearly
growing drift coefficients that are locally one-sided Lipschitz
continuous, are presented in this article. The diffusion
coefficients are assumed to be locally Lipschitz continuous and have
at most linear growth. Furthermore, the classical rate of
convergence, i.e. one--half, for such schemes is recovered  when the
local Lipschitz continuity assumptions are replaced by global and,
in addition, it is assumed that the drift coefficients satisfy
polynomial Lipschitz continuity.}

\def \one{\ \hbox{I\hskip-.60em 1}}

\begin{document}

\section{Introduction}

It is a well-known result that a stochastic differential equation
(SDE) with a superlinearly growing drift coefficient has a unique
solution if, its drift and diffusion coefficients satisfy a suitable
monotone growth condition, so-called coercivity, and a linear growth
condition respectively. Typically, suitable local Lipschitz
continuity conditions are also required of the coefficients. One
could refer to Krylov \cite{Krylov-paper} and the references therein
for more details. Moreover, the almost sure convergence and
convergence in probability of the corresponding (explicit) Euler
approximations were proved by Gy\"{o}ngy \cite{Gyongy}. However,
Hutzenthaler, Jentzen and Kloeden \cite{HJK-div} showed recently
that the absolute moments of (the aforementioned) Euler
approximations at a finite time could diverge to infinity as implied
in Higham, Mao and Stuart \cite{HMS}. In other words, the essential
property of uniform integrability may not hold for such sequences.
Thus, one could not obtain results on strong (in an $\mathcal
L^p$-sense) approximations although such results exist in the cases
of almost sure convergence and convergence in probability. One
further realises that the introduction of accelerated Monte Carlo
schemes provides a strong incentive for the study of strong
approximations of SDEs since, results on the latter are required for
the efficient implementation of the former. More information on this
topic can be found in Gile's seminal paper \cite{Giles}, Giles and
Szpruch \cite{Giles-Szpruch} and the references therein.

Recently Hutzenthaler, Jentzen and Kloeden \cite{HJK-converge}
introduced the notion of tamed Euler schemes in which the drift term
is modified so that it is uniformly bounded. With such an approach,
they are able to prove that the tamed Euler scheme converges
strongly (with rate one-half) to the exact solution of the SDE if
the drift coefficient is globally one-sided Lipschitz continuous and
has a derivative which grows (at most) polynomially. In addition,
they assume that the diffusion coefficient of the SDE satisfies a
global Lipschitz condition and it grows at most linearly.
Furthermore, they offer a detailed review of the use of implicit
schemes and compare them with tamed Euler approximations. Their
comparison demonstrates that the implementation of implicit schemes
requires significantly more computational effort than the tamed
version.

One however observes in \cite{HJK-converge} that for the ``taming"
of the drift coefficient, the term $n^{-1}$ is used when it is known
from classical literature that the standard strong convergence rate
is one-half. In other words, one expects that the use of $n^{-1/2}$
should be sufficient in order to control the drift coefficient and
achieve strong convergence of the numerical scheme. In this article,
a generalisation of the results of Hutzenthaler, Jentzen and Kloeden
\cite{HJK-converge} is presented by using a variant of their tamed
Euler method while a simpler proof is provided. It is proved that,
even when global Lipschitz continuity conditions are replaced by
local conditions, the tamed Euler scheme converges in $\mathcal L^p$
to the exact solution of the SDE. Moreover, as a consequence of the
aforementioned generalisation, the classical rate of convergence is
obtained under the same assumptions as in \cite{HJK-converge}. In
fact, one further observes that the use of $n^{-\alpha}$, where
$\alpha \in (0, 1/2]$, is also suitable  for proving $\mathcal L^p$
convergence of such tamed schemes, e.g. see Theorem \ref{mainthm}
below. Naturally, this implies that the proposed tamed
coefficients/schemes belong to a large class of functions/schemes
which satisfy certain properties. For example, \eqref{uni_bound}
from below represents a suitable condition on tamed coefficients so
as to achieve uniform moment bounds. Similarly, Hutzenthaler and
Jentzen \cite{HJ} offer results on a class of suitably ``tamed"
numerical schemes by applying space truncation techniques, e.g.
corollary 2.19 in \cite{HJ}.

We conclude this section by introducing some basic notation. The
norm of a vector $x \in \mathbb R^d$ and the Hilbert-Schmidt norm of
a matrix $A\in \mathbb R^{d\times m}$ are respectively denoted by
$|x|$ and $|A|$. The transpose of a matrix $A \in \mathbb R^{d\times
m}$ is denoted by $A^{T}$ and the scalar product of two vectors $x,
y \in \mathbb R^d$ is denoted by $xy$. The integer part of a real
number $x$ is denoted by $[x]$. Moreover, $\mathcal L^p=\mathcal
L^p(\Omega, \mathcal F, \mathbb P)$ denotes the space of random
variables $X$ with a norm $\|X\|_p:=\big(\mathbb
E\big[|X|^p\big]\big)^{1/p} < \infty$ for $p>0$. Finally,  $\mathcal
B(V)$ denotes the $\sigma$-algebra of Borel sets of a topological
space $V$.

\section{Main Result}

Let $(\Omega, \{\mathcal F_t\}_{\{t \geq 0\}}, \mathcal F, \mathbb
P)$  be a filtered probability space satisfying the usual
conditions, i.e. the filtration is increasing, right continuous and
complete. Let $\{W(t)\}_{\{t \geq 0\}}$ be an $m$-dimensional Wiener
martingale. Furthermore, it is assumed that $b(t,x)$ and
$\sigma(t,x)$ are $\mathcal B(\mathbb R_+) \otimes
\mathcal{B}(\mathbb R^d)$-measurable functions which take values in
$\mathbb R^d$ and $\mathbb R^{d \times m}$ respectively. For a fixed
$T>0$, let us consider an SDE given by
\begin{eqnarray}
dX(t)=b(t, X(t))dt+\sigma(t,X(t))dW(t), \qquad \forall  \, t\in
[0,\, T], \label{sde}
\end{eqnarray}
with initial value $X(0)$ which is an almost surely finite $\mathcal
F_0$-measurable random variable.

For every $n \geq 1$, and any $t \in [0,T]$, the
following tamed Euler scheme is defined
\begin{eqnarray}
dX_n(t)=b_n(t,X_n(\kappa_n(t)))dt+\sigma(t,X_n(\kappa_n(t)))dW(t)
\label{tem}
\end{eqnarray}
with the same initial value $X(0)$ as SDE \eqref{sde} and
$\kappa_n(t):=[nt]/n$. Moreover, it is assumed that
\begin{equation}
b_n(t,x):= \frac{1}{1+n^{-\alpha}|b(t,x)|}b(t,x),
\end{equation}
for any $t \in [0,T]$, $x \in \mathbb R^d$ and $\alpha
 \in (0, 1/2]$. One then observes that
\begin{equation} \label{uni_bound}
|b_n(t,x)|\le \min(n^{\alpha},|b(t,x)|).
\end{equation}
Moreover, for every $n \geq 1$, one deduces immediately that
$b_n(t,x)$ is a $\mathcal B(\mathbb R_+) \otimes \mathcal B(\mathbb
R^d)$-measurable function which take values in $\mathbb R^d$.

Let $\mathbb{L}^p$ denote the set of nonnegative $p$-th integrable
functions on $[0,\, T]$, i.e. to say if $f \in \mathbb{L}^p$ then
$$
\int_0^T |f(t)|^p dt < \infty.
$$
We make the following assumptions.
\newline

\textbf{A-1.} There exists a positive constant $K$ such that,
$$
2xb(t,x) \vee |\sigma(t,x)|^2 \leq K(1+|x|^2)
$$
for any $t \in [0,T]$ and $x \in \mathbb R^d$.

\textbf{A-2.} For every $R>0$, there exists a positive constant
$L_R$ such that, for any $t \in [0,T]$,
\begin{eqnarray}
 2(x-y)(b(t,x)-b(t,y)) \vee  |\sigma(t,x)-\sigma(t,y)|^2 \leq L_R |x-y|^2  \nonumber
\end{eqnarray}
 for all $|x|,  |y| \leq R$.

\textbf{A-3.} For every $R \ge 0$ and $p>0$, there exists $N_{R}\in
\mathbb{L}^p$, such that
$$
\sup_{|x|\le R}|b(t,x)| \le N_R(t)
$$
for any $t \in [0,T]$.

\textbf{A-4.} For every $p>0$, $\mathbb{E}[|X(0)|^p]<\infty$.

\begin{remark}
Note that due \eqref{uni_bound}, for each $n \ge 1$, the norm of
$b_n$ is a bounded function of $t$ and $x$ and, due to \textbf{A-1},
the norm of $\sigma$ has at most linear growth. This fact guarantees
the existence of a unique solution to \eqref{tem}. Moreover, it
guarantees that for each $n\ge 1$, all moments exist, each of which
is bounded above by some value that depends on $n$, i.e.
\begin{equation} \label{n_bound}
\sup_{0\le t\le T}\mathbb{E}[|X_n(t)|^p] \le N
\end{equation}
for any $p>0$, where $N:=N(n, p, T, \mathbb{E}|X(0)|^p)$ is a
positive constant.
\end{remark}

\begin{theorem}\label{mainthm}
Suppose \textbf{A-1} -- \textbf{A-4} hold, then the tamed Euler
scheme \eqref{tem} converges to the true solution of SDE \eqref{sde}
in $\mathcal L^p$-sense, i.e.
\begin{eqnarray}
\lim_{n \rightarrow \infty}\mathbb E \Bigg[ \sup_{0 \leq t \leq T}|X(t)-X_n(t)|^p\Bigg]=0 \notag
\end{eqnarray}
for all $p >0 $.
\end{theorem}

\textbf{A-5.} There exists positive constants $l$ and $L$ such that,
for any $t \in [0,T]$,
\begin{eqnarray}
 (x-y)(b(t,x)-b(t,y)) \vee  |\sigma(t,x)-\sigma(t,y)|^2 \leq L |x-y|^2  \nonumber
\end{eqnarray}
and
\begin{eqnarray}
 |b(t,x)-b(t,y)|  \leq L(1+ |x|^l + |y|^l) |x-y|  \nonumber
\end{eqnarray}
for all $x,\, y \in \mathbb R^d$.

\begin{corollary}\label{maincor}
Suppose \textbf{A-1} and \textbf{A-3}--\textbf{A-5} hold, then the
tamed Euler scheme \eqref{tem} with $\alpha = 1/2$ converges to the
true solution of SDE \eqref{sde} in $\mathcal L^p$-sense with order
1/2, i.e.
\begin{eqnarray}
\mathbb E \Bigg[ \sup_{0 \leq t \leq T}|X(t)-X_n(t)|^p\Bigg] \le C
n^{-p/2} \notag
\end{eqnarray}
for all $p >0 $, where $C$ is a constant independent of $n$.
\end{corollary}

\section{Moment bounds}

\begin{lemma}
\label{l1} Consider the tamed Euler scheme given by equation
(\ref{tem}). If for some $p\ge 2$,
$$
\sup_{n\ge 1}\sup_{0\le t\le T}\mathbb{E}[|X_n(t)|^p]<\infty
$$
and \textbf{A-1} hold, then
\begin{equation}
\label{lem1} \sup_{0\le t \le T}\mathbb E [ |X_n(t
)-X_n(\kappa_n(t))|^p] \le C n^{-p/2}
\end{equation}
and
\begin{equation}
\label{lem2} \sup_{0\le t \le T}\mathbb E\Big[|X_n(t
)-X_n(\kappa_n(t))|^p|b_n(t, X_n(\kappa_n(t)))|^p\Big] \le C,
\end{equation}
where $C$ is a positive constant independent of $n$.
\end{lemma}

\begin{proof}
One immediately writes
\begin{align}
\mathbb E |X_n(t)-X_n(\kappa_n(t))|^p  =& \mathbb E|
\int_{\kappa_n(t)}^{t}b_n(r,
X_n(\kappa_n(r)))dr+\int_{\kappa_n(t)}^{t}\sigma(r,
X_n(\kappa_n(r)))dW(r) |^p \nonumber
\end{align}
for every $t\in [0,\,T]$, and thus, due to H\"older's inequality,
\begin{align}
\mathbb E  |X_n(t)-X_n(\kappa_n(t))|^p  \le & 2^{p-1}
|t-\kappa_n(t)|^{p-1}\mathbb{E}\int_{\kappa_n(t)}^t |b_n(r,
X_n(\kappa_n(r)))|^p dr \nonumber
\\
&  +2^{p-1}  \mathbb E |\int_{\kappa_n(t)}^t \sigma(r,
X_n(\kappa_n(r)))dW(r) |^p.    \label{g1}
\end{align}
One then observes that,
\begin{align} \label{g2}
2^{p-1} |t-\kappa_n(t)|^{p-1}\mathbb E\int_{\kappa_n(t)}^t |b_n(r,
X_n(\kappa_n(r)))|^p dr &
%\le \Big(\frac{2}{n}\Big)^{p-1} \mathbb
%E\int_{\kappa_n(t)}^t n^{\alpha p} dr
\le 2^{p-1}n^{(\alpha - 1)p}
\end{align}
and since \textbf{A-1} holds and $\sup_{n\ge1}\sup_{t\le T}\mathbb{E}[|X_n(t)|^p]<\infty$, for some $p\ge 2$, then
\begin{align} \label{g3}
 \mathbb E\bigg[\bigg |\int_{\kappa_n(t)}^t \sigma(r, X_n(\kappa_n(r)))dW(r) \bigg|^p\bigg]
  \le & C
 E\bigg[\int_{\kappa_n(t)}^t (1+ |X_n(\kappa_n(r))|^2)dr\bigg]^{p/2} \le C n^{-p/2}, 
\end{align}
where $C$ denotes some positive (general) constant which is independent of $n$ and $t$. Substituting \eqref{g2} and \eqref{g3} in \eqref{g1} yields
\eqref{lem1}. Furthermore, \eqref{lem2} holds trivially, since
\begin{align}
\mathbb E \Bigg[ |X_n(t)-X_n(\kappa_n(t))|^p|b_n(t,
X_n(\kappa_n(t)))|^p \Bigg] & \le \mathbb E \Bigg[ |X_n(t
)-X_n(\kappa_n(t))|^p \Bigg]n^{\alpha p}  \le C \nonumber
\end{align}
is true, for any $t\in [0,\,T]$, due to \eqref{lem1}.
\end{proof}

\begin{lemma}
\label{l2_moment_bound} Suppose that  \textbf{A-1} and \textbf{A-4}
hold, then for some $C:=C(T,\, K, \, \mathbb{E}[|X(0)|^2])$,
\begin{equation} \label{l2_bound}
 \sup_{n \ge 1}\sup_{0\leq t \leq T}\mathbb E
\bigg[ |X_n(t)|^2\bigg] <C.
\end{equation}
\end{lemma}
\begin{proof}
Let us  define
$$
I_n(T):=\mathbb{E}\Big[\int_0^T (X_n(s) - X_n(\kappa_n(s)) b_n(s,
X_n(\kappa_n(s))) ds\Big].
$$
Then, one calculates
\begin{align*}
I_n(T) &  =  \mathbb{E}\Big[\int_0^T
\Big(\int_{\kappa_n(s)}^{s}b_n(r, X_n(\kappa_n(r)))dr
 +\int_{\kappa_n(s)}^{s}\sigma(r,
X_n(\kappa_n(r)))dW(r)\Big) b_n(s, X_n(\kappa_n(s))) ds\Big]  \\
& =  \sum_{k=0}^{n([T]+1)}\int_{\frac{k}{n}}^{\frac{k+1}{n} \wedge
T}\mathbb{E}\Big[b_n(s, X_n(k/n))\mathbb{E}
\Big(\int_{\frac{k}{n}}^{s}b_n(r, X_n(k/n))dr
 +\int_{\frac{k}{n}}^{s}\sigma(r,
X_n(k/n))dW(r)\Big|\mathcal{F}_{\frac{k}{n}}\Big)\Big]
ds \\
& = \mathbb{E}\Big[\int_0^T b_n(s,
X_n(\kappa_n(s)))\int_{\kappa_n(s)}^{s}b_n(r, X_n(\kappa_n(r)))dr
ds\Big] \nonumber
\end{align*}
and thus
\begin{align} \label{I}
|I_n(T)|  \le & \mathbb{E}\Big[\int_0^T |b_n(s,
X_n(\kappa_n(s)))|\int_{\kappa_n(s)}^{s}|b_n(r, X_n(\kappa_n(r)))|dr
ds\Big] \le T n^{2\alpha -1} \le T.
\end{align}
Furthermore, It\^{o}'s formula gives
\begin{align} \label{Ito_2}
 |X_n(t)|^2  = & |X(0)|^2 + 2 \int_0^t X_n(s) b_n(s, X_n(\kappa_n(s))) ds + \int_0^t|
\sigma(s, X_n(\kappa_n(s)))|^2 ds  \nonumber \\ & + 2 \int_0^t
X_n(s) \sigma(s, X_n(\kappa_n(s))) dW(s) \nonumber
\\  \le & |X(0)|^2 +  2\int_0^t X_n(\kappa_n(s))b_n(s,
X_n(\kappa_n(s))) ds + \int_0^t | \sigma(s, X_n(\kappa_n(s)))|^2ds \nonumber \\
& + 2 \int_0^t (X_n(s) - X_n(\kappa_n(s))b_n(s, X_n(\kappa_n(s))) ds
+ 2 \int_0^t X_n(s)\sigma(s, X_n(\kappa_n(s))) dW(s)
\end{align}
and thus, due to \textbf{A-1}, \eqref{n_bound} and \eqref{I}, for
any $t\in [0,\,T]$,
\begin{align}
\mathbb{E}|X_n(t)|^2  \le &  C (1 + \mathbb{E}|X(0)|^2 +
\mathbb{E}\int_0^t
|X_n(\kappa_n(s))|^2 ds)   \nonumber \\
\le &  C (1 + \mathbb{E}|X(0)|^2 + \int_0^t \sup_{0\le u\le
s}\mathbb{E}|X_n(u)|^2 ds) \nonumber
\end{align}
which implies,
\begin{align}
\sup_{0\le u\le t}\mathbb{E}|X_n(u)|^2 & \le  C (1 +
\mathbb{E}|X(0)|^2 + \int_0^t \sup_{0\le u\le s}\mathbb{E}|X_n(u)|^2
ds) <\infty \nonumber
\end{align}
where the positive (general) constant $C$ is independent of $n$. One
then observes that the application of Gronwall's lemma yields
$$
\sup_{0\le u\le T}\mathbb{E}|X_n(u)|^2 <C
$$
where $C:=C(T,\,K, \, \mathbb{E}[|X(0)|^2])$.
\end{proof}

\begin{lemma}
\label{moment_bound} Suppose that  \textbf{A-1} and \textbf{A-4}
holds, then for some $C:=C(p,\,T,\,K, \,  \mathbb{E}[|X(0)|^p])$,
\begin{equation} \label{bound}
 \mathbb E \bigg[ \sup_{0\leq t \leq
T}|X(t)|^p\bigg] \vee \sup_{n \ge 1}\mathbb E \bigg[ \sup_{0\leq t
\leq T}|X_n(t)|^p\bigg] <C
\end{equation}
for every $p >0$.
\end{lemma}

\begin{proof} It is well known in the literature that the result
$$
\mathbb E \bigg[ \sup_{0\leq t \leq T}|X(t)|^p\bigg] < C
$$
holds for every $p >0$. One could consult, for example, Krylov
(1980) for more details.
\newline

In order to prove the second part of (\ref{bound}), an inductive argument is used below. First, one chooses $p= 2$ and
observes that due to Lemma \ref{l2_moment_bound} that
$$
\sup_{n \ge 1}\sup_{0\le t\le T}\mathbb{E}|X_n(t)|^2< C
$$
holds for some positive constant $C:= C(T, K, \mathbb{E}[|X(0)|^2])$
which is independent of $n$. Thus, \eqref{lem2} from Lemma \ref{l1}
holds true for $p=2$ and one could use \eqref{Ito_2} to obtain the
following estimate for $q=2p$, i.e. $q=4$,
\begin{align} \label{q-sup-estimate}
 \mathbb{E}[\sup_{0\le s\le t}|X_n(s)|^q] \le & C(1 +
 \mathbb{E}[|X(0)|^q] +  \int_0^t \mathbb{E}|X_n(\kappa_n(s))|^q ds
\nonumber \\ & +  \int_0^t \mathbb{E}[|X_n(s) -
X_n(\kappa_n(s))|^{q/2} |b_n(s, X_n(\kappa_n(s)))|^{q/2} ]ds
 \nonumber \\ & + \mathbb{E}[\sup_{0\le s\le t}|\int_0^s X_n(u)\sigma(u, X_n(\kappa_n(u)))
dW(u)|^{q/2}])
\end{align}
and the application of the Burkholder-Davis-Gundy (BDG) inequality
yields
\begin{align}
 \mathbb{E}[\sup_{0\le s\le t}|X_n(s)|^q] \le & C\{1 +
 \mathbb{E}[|X(0)|^q] +  \int_0^t \mathbb{E}[\sup_{0\le u \le s}|X_n(u)|^q]ds
 \nonumber \\ & + \mathbb{E}[(\int_0^t |X_n(s)|^2 |\sigma(s,
 X_n(\kappa_n(s)))|^2
ds)^{q/4}]\}, \nonumber
\end{align}
where $C$ denotes again a general constant which is independent of
$n$. Thus, the application of Young's inequality yields
\begin{align}
 \mathbb{E}[\sup_{0\le s\le t}|X_n(s)|^q] \le & C\{1 +
 \mathbb{E}[|X(0)|^q] +  \int_0^t \mathbb{E}[\sup_{0\le u \le s}|X_n(u)|^q ]ds +
 \frac{1}{2C}\mathbb{E}[\sup_{0\le s\le t}|X_n(s)|^q]
 \nonumber \\ & + \frac{C}{2}\mathbb{E}[(\int_0^t |\sigma(s, X_n(\kappa_n(s)))|^2 ds)^{q/2}]\} \nonumber
\end{align}
which, due to \textbf{A-1} and H\"{o}lder's inequality implies that
\begin{align}
 \mathbb{E}[\sup_{0\le u\le t}|X_n(u)|^q] \le & C(1 +
 \mathbb{E}[|X(0)|^q] +  \int_0^t \mathbb{E}[\sup_{0\le u\le s}|X_n(u)|^q ]ds) < \infty \nonumber
\end{align}
and thus the application of Gronwall's lemma yields that
\begin{equation} \label{q-estimate}
\mathbb{E}[\sup_{0\le t\le T}|X_n(t)|^q]< C
\end{equation}
holds for some positive constant $C:= C(q, T, K, \mathbb{E}|X(0)|^q)$
which is independent of $n$. Thus, \eqref{lem2} from Lemma \ref{l1}
holds true for $p=4$ and one could use \eqref{q-sup-estimate} to
obtain the estimate \eqref{q-estimate} for $q=2p$, i.e. $q=8$.
Repeating the same procedure (by induction) one obtains the desired
result \eqref{bound}.
\end{proof}

\section{Proof of Main Result}

For every $R>0$ and $n\ge 1$, let us consider the stopping times
\begin{equation}
\tau_R:=\inf\{t \ge 0: |X(t)| \geq R\}, \, \rho_{nR}:=\inf\{t \ge
0:|X_n(t)| \ge R\} \text{ and } \quad \nu_{nR}:=\tau_R \wedge
\rho_{nR}. \label{nmc}
\end{equation}
\begin{lemma}
\label{b_n_to_b}
Suppose that \textbf{A-3} holds, then for any $R>0$ and $p>0$
\begin{equation} \label{difference_b}
\lim_{n \to \infty}\mathbb{E}\Big[\int_0^T |b(s\wedge \nu_{nR},
X_n(\kappa_n(s \wedge \nu_{nR})))- b_n(s \wedge \nu_{nR},
X_n(\kappa_n(s \wedge \nu_{nR})))|^p ds\Big]=0.
\end{equation}
\end{lemma}
\begin{proof}
One immediately observes that for $p\ge2$
\begin{align} \label{difference_rate}
&\mathbb{E}\Big[\int_0^T |b(s\wedge \nu_{nR}, X_n(\kappa_n(s \wedge
\nu_{nR})))-b_n(s \wedge \nu_{nR}, X_n(\kappa_n(s \wedge
\nu_{nR})))|^pds\Big] \nonumber \\ \le & n^{-\alpha
p}\mathbb{E}\Big[\int_0^T \frac{|b(s \wedge \nu_{nR}, X_n(\kappa_n(s
)\wedge \nu_{nR}))|^{2p}}{(1+n^{-\alpha}|b(s \wedge \nu_{nR},
X_n(\kappa_n(s \wedge \nu_{nR})))|)^p}ds\Big] < \infty
\end{align}
due to \textbf{A-3}. Thus the application of the dominated
convergence theorem  yields the desired result.
\end{proof}

\begin{proof}[\textbf{ Proof of theorem \ref{mainthm}.}]
Let $p\ge2$ and consider
$$
\chi_n(s):=X(s\wedge \nu_{nR})-X_n(s\wedge \nu_{nR}).
$$
One observes immediately that
\begin{align} \label{one}
\mathbb E\Bigg[\sup_{0 \le t \le T}|X(t)-X_n(t)|^p \Bigg] \leq &
\mathbb E \Bigg[ \sup_{0 \le t \le T}|X(t)-X_n(t)|^p \one_{\{\tau_R
\le T \mbox{ or } \rho_{nR} \le T\}}\Bigg] + \mathbb E \Bigg[
\sup_{0 \le t \le T}|\chi_n(s)|^p \Bigg].
\end{align}
Then, by the application of Young's inequality for $q
> p$ and $\eta>0$ one obtains
\begin{align} \label{two2}
\mathbb E \Bigg[ \sup_{0 \leq t \leq T}|X(t)-X_n(t)|^p
\one_{\{\tau_R \le T \mbox{ or } \rho_{nR} \le T \}}\Bigg] \le &
\frac{ \eta p }{q}\mathbb E \Bigg[ \sup_{0 \le t \le
T}|X(t)-X_n(t)|^q \Bigg] \nonumber \\ & +\frac{q-p}{q
\eta^{p/(q-p)}} \mathbb P(\tau_R \le T \mbox{ or } \rho_{nR} \leq T
) \nonumber
\\
\le & \frac{ \eta p}{q} 2^{q} C   + \frac{q-p}{q \eta^{p/(q-p)}}
\Bigg \{\mathbb E\Bigg[\frac{|X(\tau_R)|^p}{R^p}\Bigg] + \mathbb
E\Bigg[\frac{|X_n(\rho_{nR})|^p}{R^p}\Bigg]\Bigg\} \nonumber \\
\le & \frac{ \eta p}{q} 2^{q} C  + \frac{q-p}{q \eta^{p/(q-p)}R^p}
2C.
\end{align}
Furthermore, one defines
$$
\beta_n(s):= \Big( b(s, X(s))-b_n(s,X_n(\kappa_n(s)))\Big) \one_{[s
\le \nu_{nR}]}
$$
and
$$\alpha_n(s):=\Big(\sigma(s,X(s))-\sigma(s,X_n(\kappa_n(s)))\Big) \one_{[s \le
\nu_{nR}]}$$ to obtain
\begin{align} \label{xi}
|\chi_n(t)|^2 =  \int_{0}^t\Big[ 2\chi_n(s)\beta_n(s) +
|\alpha_n(s)|^2 \Big] ds  + 2 \int_{0}^t\chi_n(s)\alpha_n(s) dW(s)
\end{align}
with
\begin{align} \label{xi_drift}
\chi_n(s)\beta_n(s)  = & \Big\{(X(s)-X_n(\kappa_n(s))(b(s,
X(s))-b(s,X_n(\kappa_n(s)))) \nonumber
\\
& + (X(s)-X_n(\kappa_n(s)))(b(s,
X_n(\kappa_n(s)))-b_n(s,X_n(\kappa_n(s)))) \nonumber
\\ & + (X_n(\kappa_n(s)) - X_n(s))(b(s,
X(s))-b_n(s,X_n(\kappa_n(s))))\Big\} \one_{[s \le \nu_{nR}]}
\end{align}
which implies, due to \textbf{A-2} and \textbf{A-3},
\begin{align} \label{J}
\chi_n(s)\beta_n(s)\le J_n(s): = & \Big\{(2L_R+1)| \chi_n(s)|^2 +
(2L_R+1)|X_n(s)-X_n(\kappa_n(s))|^2 \nonumber
\\
& +|b(s, X_n(\kappa_n(s)))-b_n(s,X_n(\kappa_n(s)))|^2 \nonumber
\\
& + 2N_{R}(t)| X_n(s)-X_n(\kappa_n(s))|\Big\}\one_{[s \le
\nu_{nR}]},
\end{align}
whereas
\begin{align} \label{alpha}
|\alpha_n(s)|^p \le & 2^{p-1}\Big\{L^{p/2}_R| \chi_n(s)|^p +
L^{p/2}_R|X_n(s)-X_n(\kappa_n(s))|^p \Big\}\one_{[s \le \nu_{nR}]}.
\end{align}
Furthermore, in view of the above estimate \eqref{alpha}, one
observes that the application of Young's inequality yields
\begin{align} \label{xi_diffussion}
|\chi_n(s)\alpha_n(s)|^{p/2} \le & 2^{p/2}\Big\{(2L_R^{p/4}+1)|
\chi_n(s)|^p + L^{p/4}_R|X_n(s)-X_n(\kappa_n(s))|^p \Big\}\one_{[s
\le \nu_{nR}]}.
\end{align}
Thus, from \eqref{xi} one obtains by applying H\"{o}lder and BDG
inequalities that
\begin{align}
\mathbb{E}[\sup_{0\le u \le t}|\chi_n(u)|^p] \le
C\mathbb{E}\int_{0}^t \Big[ |J_n(s)|^{p/2} + |\alpha_n(s)|^p +
|\chi_n(s)\alpha_n(s)|^{p/2}\Big]  ds, \nonumber
\end{align}
where $C:=C(p,\,T)$ is a positive constant, which in view of
\eqref{xi_drift}, \eqref{alpha} and \eqref{xi_diffussion} yields
\begin{align} \label{mean_xi}
\mathbb{E}[\sup_{0\le u \le t}|\chi_n(u)|^p] \le & C_R\int_{0}^t
\Big[ \mathbb{E}|\chi_n(s)|^p + \mathbb{E}|X_n(s\wedge
\nu_{nR})-X_n(\kappa_n(s)\wedge \nu_{nR})|^p \Big] ds, \nonumber \\
& + C\Big[ \mathbb{E}\int_{0}^{t\wedge \nu_{nR}}|b(s,
X_n(\kappa_n(s)))-b_n(s,X_n(\kappa_n(s)))|^pds \nonumber \\
 &+ \int_0^t N^{p/2}_{R}(s)\mathbb{E}| X_n(s\wedge \nu_{nR})-X_n(\kappa_n(s)\wedge
 \nu_{nR})|^{p/2} ds \Big],
\end{align}
where $C_R:=C_R(p,\,T,\, L_R)$  and (the redefined) $C:=C(p,\,T)$ are
positive consants. The application of Grownwall inequality results
in
$$
\lim_{n\to \infty }\mathbb{E}[\sup_{0\le t \le T}|\chi_n(t)|^p]=0
$$
for every $R>0$ due to Lemmas \ref{l1} and \ref{b_n_to_b}. Finally,
given an $\epsilon
>0$, one can choose $\eta$ small enough so
$$
\frac{ \eta p}{q} 2^{q} C < \frac{\epsilon}{3},
$$
$R$ large enough so
$$
\frac{q-p}{q \eta^{p/(q-p)}R^p} 2C < \frac{\epsilon}{3}
$$
and $n$ large enough so
$$
\mathbb{E}[\sup_{0\le t \le T}|\chi_n(t)|^p]< \frac{\epsilon}{3}
$$
to obtain due to \eqref{one} and \eqref{two2} that
$$
\mathbb E\Bigg[\sup_{0 \le t \le T}|X(t)-X_n(t)|^p \Bigg] < \epsilon
$$
and thus prove the desired result.
\end{proof}

\section{Rate of Convergence}
First one observes that if \textbf{A-4} and \textbf{A-5} hold, then
\begin{align} \label{poly_growth}
|b(t,x)| \le |b(t,x) - b(t, 0)| + |b(t,0)| \le L(1 + |x|^l)|x| +
N_0(t) \le N(t)(1 + |x|^{l+1})
\end{align}
for any $t\in [0,\,T]$ and $x\in\mathbb{R}^d$, where $N(t) \in
\mathbb{L}^p$ for any $p>0$.

\begin{proof}[\textbf{ Proof of Corollary \ref{maincor}.}]
First one rewrites \eqref{xi_drift} in the following way
\begin{align} \label{new_xi_drift}
\chi_n(s)\beta_n(s)  = & \Big\{(X(s)-X_n(s))(b(s, X(s))-b(s,X_n(s)))
\nonumber \\ & + (X(s)-X_n(s))(b(s, X_n(s))-b(s,X_n(\kappa_n(s))))
\nonumber \\ & + (X(s) - X_n(s)) (b(s,X_n(\kappa_n(s)))-
b_n(s,X_n(\kappa_n(s))))\Big\} \one_{[s \le \nu_{nR}]}
\end{align}
and adjusts accordingly, due to \textbf{A-5}, the upper bound $J_n$
from \eqref{J}
\begin{align} \label{new_J}
\chi_n(s)\beta_n(s)\le J_n(s): = & \Big\{(L+1)| \chi_n(s)|^2 +
L^2(1+ |X_n(s)|^{l} +
|X_n(\kappa_n(s))|^{l})^2|X_n(s)-X_n(\kappa_n(s))|^2 \nonumber
\\
& +|b(s, X_n(\kappa_n(s)))-b_n(s,X_n(\kappa_n(s)))|^2\Big\} \one_{[s
\le \nu_{nR}]}
\end{align}
and, thus, the last term of \eqref{mean_xi} is replaced by
\begin{align}
\mathcal{E}(t): = & \mathbb{E}\Big[\int_0^{t\wedge \nu_{nR}}  C(1+
|X_n(s)|^{lp} + |X_n(\kappa_n(s))|^{lp})|X_n(s)-X_n(\kappa_n(s))|^p
\Big]ds \nonumber
\end{align}
which is estimated from above by
\begin{align}
\mathcal{E}(t) \le C \int_0^{t} \Big(\sqrt{\mathbb{E}|X_n(s\wedge
\nu_{nR})-X_n(\kappa_n(s)\wedge \nu_{nR})|^{2p}} \Big)ds \nonumber
\end{align}
due to H\"{o}lder's inequality and \eqref{bound}. Note that the general constant $C$ is independent of
$t$ and $n$. In view of Lemma \ref{l1}, one deduces that
\begin{align} \label{E}
\sup_{0\le t\le T}\mathcal{E}(t) \le C n^{-p/2}.
\end{align}
Furthermore, \eqref{difference_rate}, \eqref{bound} and  \eqref{poly_growth} imply
that
\begin{align} \label{rate_b_b_n}
\mathbb{E}[\int_0^T |b(s\wedge \nu_{nR}, X_n(\kappa_n(s \wedge
\nu_{nR})))-b_n(s \wedge \nu_{nR}, X_n(\kappa_n(s \wedge
\nu_{nR})))|^p ]\le C n^{-p/2}
\end{align}
which along with \eqref{lem1}, \eqref{E} and \eqref{rate_b_b_n} result in
\begin{align}\label{rate_xi}
\mathbb{E}[\sup_{0\le t \le T}|\chi_n(t)|^p] \le C n^{-p/2}
\end{align}
due to \eqref{mean_xi}. Finally, one chooses $\eta=n^{-\frac{p}{2}}$,
$R=n^{\frac{q}{2(q-p)}}$, $q>p \ge 2 $, to obtain the desired result
due to \eqref{one}, \eqref{two2} and \eqref{rate_xi}.

\end{proof}

\end{document}